\documentclass[11pt,leqno]{amsart}
\usepackage[foot]{amsaddr}
\usepackage[margin=1.45in]{geometry}
\usepackage[colorlinks=true,pdfborder={0 0 0}]{hyperref}
\hypersetup{
	colorlinks=true,       % false: boxed links; true: colored links
	linkcolor=red,          % color of internal links
	citecolor=blue,        % color of links to bibliography
	filecolor=magenta,      % color of file links
	urlcolor=cyan           % color of external links
}
\usepackage{amsmath}
\usepackage{amsfonts,amssymb}
\usepackage{dsfont}
\usepackage{enumerate}
\usepackage{mathrsfs}
\usepackage{amsthm}
\usepackage{xcolor}
\usepackage{tikz}
\usepackage{textcomp}
\usepackage{ucs}
\usepackage{textcomp}
\usepackage{dsfont}
\theoremstyle{plain}
\newtheorem{theo}{Theorem}[section]
\newtheorem*{theo*}{Theorem}
\newtheorem{prop}[theo]{Proposition}
\newtheorem{lemm}[theo]{Lemma}
\newtheorem{coro}[theo]{Corollary}

\theoremstyle{definition}

\DeclareMathOperator{\diff}{d}
\DeclareSymbolFont{pletters}{OT1}{cmr}{m}{sl}
\DeclareMathSymbol{s}{\mathalpha}{pletters}{`s}

\def\ba{\begin{align}}
	\def\bad{\begin{aligned}}
		\def\be{\begin{equation}}
			\def\ea{\end{align}}
		\def\ead{\end{aligned}}
	\def\ee{\end{equation}}

\def \D{\diff\! }

\def\dxi{\diff \! \xi}

\def\dx{\diff \! x}

\def\le{\leq}

\def \R{\mathbb{R}}

\numberwithin{equation}{section}

\title{Mathematical model for collective migration on a viscoelastic collagen network}

\author{Nicolas Meunier} \address[N. Meunier]{LaMME, UMR CNRS
  8071, Universit\'e d'\'Evry, Val d'Essonne, France}
\email{nicolas.meunier@univ-evry.fr}

\author{Andrei Tarfulea} \address[A. Tarfulea]{Department of Mathematics, Louisiana State University}
\email{tarfulea@lsu.edu}

\begin{document}
	
\maketitle

\begin{abstract}
  In this paper, we study a model of self-generated directional cell migration on viscoelastic substrates in the absence of apparent intrinsic polarity. This model, first proposed in \cite{Clark}, was observed numerically to manifest traveling pulse solutions for sufficiently large collagen stiffness, leading to a persistent collective migration. Here we provide a rigorous mathematical framework for the model, finding the exact stationary states and conditional traveling pulse. We also prove global well-posed in $W^{k,\infty}$ spaces, local stability of the traveling pulse for high stiffness, and exponential convergence to the stationary state for low stiffness.
\end{abstract}

%\tableofcontents
	
\section{Introduction}

In \cite{Clark} the following model was proposed for self-generated directional migration of cell clusters on viscoelastic substrates in the absence of internal biochemical polarity signals:
\begin{eqnarray}\label{eq:Clark_1}
	\alpha \partial_t S(t,x) &=& - S(t,x) + \beta ^2 \partial^2_{xx} S(t,x) + \alpha \gamma \delta_{x=x_c(t)}, \quad t>0, \ x \in \R,\\
\label{eq:Clark_2}
	\dot x_c (t)&=&\frac{\D }{\D t } x_c(t) = - \eta \partial_{x} S(t,x=x_c(t)),
\end{eqnarray}	
where $\alpha, \beta, \gamma$, $\eta$ are non-negative real numbers and  $\delta_{x=a}$ is the Dirac mass in $x=a$. The motion of the cell cluster, represented by the curve $x_c(t)$, moves according to the local deformation of the collagen network, represented by the scalar $S$. This deformation spreads and dissipates through the medium, but is also generated from a source centered at the cell cluster in an isotropic and apolar manner; i.e., the equation above is symmetric in translations and in switching $x$ to $-x$.

In \cite{Clark}, heuristic results had been given concerning the existence of a traveling pulse type solution for \eqref{eq:Clark_1} -- \eqref{eq:Clark_2}. Specifically, for a threshold $\eta_0$ that depends on $\alpha$, $\beta$, and $\gamma$, traveling pulses appear spontaneously when $\eta > \eta_0$
but do not appear when $\eta < \eta_0$.

There are two difficulties in mathematically analyzing \eqref{eq:Clark_1} -- \eqref{eq:Clark_2}. The first is due to its subtle non-linear nature. The second concerns the fact that in the traveling frame, the only way to accommodate the Dirac mass is to have a jump in $\partial_x S$ at $x=0$.  That is, $\beta^2(\partial_x S(0^-) - \partial_x S(0^+)) = - \alpha \gamma $.  On the other hand, the condition on $\dot x_c$ indicates that the velocity $c$ of the traveling pulse satisfies $c = \xi \partial_x S(0)$; however, $\partial_x S(0)$ doesn't make sense because of this jump.

To overcome the second difficulty, here we study a regularized version of \eqref{eq:Clark_1} -- \eqref{eq:Clark_2} taking the form
\be\label{eq:reg}
 \left(\frac{1}{\beta ^2} + \frac{\alpha}{\beta^2} \partial_t -\partial^2_{xx}\right) S(t,x)= \frac{\alpha \gamma }{\beta^2} g_\varepsilon (x-x_c(t)), \quad t>0, \ x \in \R,
 \ee
 where 
 %$\alpha, \beta, \gamma$ are non-negative real numbers and 
 $g_\varepsilon$ is a Gaussian like function:
 \be\label{def:gaussian}
 g_\varepsilon (x)= \frac{1}{\varepsilon \sqrt{\pi}}e^{- \frac{x^2}{\varepsilon^2}},
 \ee
 and the position $x_c(t)$ is again given as
 \be\label{def:v_c}
 \dot x_c (t)= v_c(t) = - \eta \partial_x S(t,x=x_c(t)).
 \ee
% with $\eta >0$.

Since 
\[
\lim_{\varepsilon\to 0} g_\varepsilon (x)= \delta_{x=0},
\]
in the distributional sense and 
\[
\hat g_\varepsilon (\xi) =  e^{-\frac{\varepsilon \xi^2}{4}} \le 1,
\] 
this justifies that the problem \eqref{eq:reg} -- \eqref{def:gaussian} -- \eqref{def:v_c} is an approximation of the problem \eqref{eq:Clark_1} -- \eqref{eq:Clark_2}. Throughout this work we try to make as few assumptions as possible about the production term $g_\varepsilon$. While the
Gaussian will satisfy all of our assumptions, we remark that the main features of $g_\epsilon$ will be its nonnegativity, symmetry about the origin,
regularity (i.e., $W^{k,\infty}$), and fast decay as $|x|\rightarrow \infty$.

We now summarize our main results. First, the model is globally well-posed in $L^\infty$ based spaces, with smoothness controlled by the
smoothness of $g$ (and for uniform results also the smoothness of the initial data $S_0$).

\begin{theo}\label{THM1}
Let $k \geq 2$ and assume that $S_0(x) \in W^{k,\infty}$ and $x_0 \in \mathbb{R}$ are given. Assume further that $g_\epsilon \in W^{k+1,\infty}$. Then
there exists a unique strong solution pair $(S(t,x), x_c(t))$ to \eqref{eq:reg} -- \eqref{def:gaussian} -- \eqref{def:v_c} starting from initial
data $(S_0(x), x_0)$. Here \emph{strong solution} means that for every integer $0 \leq m \leq \lfloor k/2 \rfloor$, we have $\partial_t^m S \in L^\infty_t W^{k-2m,\infty}_x$, and that $x_c \in W^{\lceil k/2 \rceil,\infty}(\mathbb{R}_+)$, with bounds that are uniform in time.\\
If $g_\varepsilon \in C^\infty$, then $(S,x_c)$ is a classical solution (i.e., $C^\infty$ in $t$ and $x$).
\end{theo}

Second, assuming we are working in the class of strong solutions given by Theorem \ref{THM1}, we have stable traveling left/right pulses when the
stiffness $\eta$ is sufficiently large, as well as stable stationary states when the stiffness is sufficiently small.

\begin{theo}\label{THM2}
\
\begin{enumerate}[i]
\item For a given starting location $x_0$, there exists a unique right-moving traveling pulse solution $S(t,x) = \bar{S}(x - x_0 - v_c t)$
and $x_c(t) = x_0 + v_c t$ to \eqref{eq:reg} -- \eqref{def:gaussian} -- \eqref{def:v_c} if and only if the stiffness $\eta$ is sufficiently large, depending
on $\alpha$, $\beta$, $\gamma$ and $g_\varepsilon$; the velocity $v_c$ can be computed implicitly from these quantities. This is made precise
in Lemma \ref{LEM:Pulse}.\\
\item When the traveling pulse exists, it is exponentially stable for sufficiently small perturbations. This is stated precisely in Proposition \ref{prop:stable_pulse}
and Corollary \ref{cor:higher_decay}.\\
\item For a given starting location $x_0$, there always exists a unique stationary solution $S(t,x) = \bar{S}_0(x-x_0)$.\\
\item When the stiffness $\eta$ is sufficiently small, all nontrivial solutions of \eqref{eq:reg} -- \eqref{def:gaussian} -- \eqref{def:v_c} converge
exponentially to $\bar{S}_0(x-\tilde{x})$ for some $\tilde{x}$. This is stated precisely in Proposition \ref{prop:stable_stat} and Corollary \ref{cor:stable_stat}.
\end{enumerate}
\end{theo}

This work is organized as follows. We comment on the existing literature concerning \eqref{eq:reg} -- \eqref{def:gaussian} -- \eqref{def:v_c} in Section \ref{sec:bio} and in Section \ref{sec:notation} we describe our notation. We study the stationary states and traveling pulses, proving parts (i) and (iii) of Theorem \ref{THM2} in Section \ref{sec:stat}. We formulate these two results in terms that include all physical constants, however the finer mathematical analysis benefits from simplifying the equation. As such, Section \ref{sec:rescaled} transforms equation \eqref{eq:reg} using rescaled variables to eliminate some of the physical constants, in order to better understand which parameters control the qualitative behaviour of the solutions. We then prove Theorem \ref{THM1} in Section \ref{sec:well}. Finally, we study the stability of the traveling pulse in Section \ref{sec:stab_tw} (proving part (ii) of Theorem \ref{THM2}), and we prove part (iv) of Theorem \ref{THM2} in Section \ref{sec:stab_stat}.

\section{Related problems}\label{sec:bio}

Self-generated gradients have recently attracted a lot of attention in the biological community. These gradients offer a robust strategy for a group of cells to orient themselves and navigate
an otherwise homogeneous medium.

We first mention phoresis which is a phenomenon in which a small, inert particle in a fluid moves relative to the undisturbed local velocity of the fluid that would exist at the point currently occupied by the particle if the particle were not present.  In this phenomenon, the driving force comes from an inhomogeneity in certain attributes of the particle-free fluid, such as temperature, pressure or fluid density. In \cite{Moran_Posner}, the authors review the rich physics underlying the operation of phoretically active colloids, and they describe their interactions and collective behaviors.  Micro- and nanoparticles can move by phoretic effects in response to externally imposed gradients of scalar quantities, such as chemical concentration or electrical potential. In such a framework, one class of active colloids can propel themselves in aqueous media by generating local gradients of concentration and electrical potential via surface reactions. Phoretic active colloids can be controlled by external stimuli and can mimic the collective behaviors exhibited by many biological swimmers. There is much in common between phoretic models and the model we study here: the latter is a minimal toy model of the phoretic type. To the best of our knowledge, there are no rigorous studies on the appearance of progressive waves in this type of experiment, and our work is a first step in this direction.

In a different direction, we mention \cite{griette2023speed} in which the authors study the velocity of the displacement wave for Fisher-KPP fronts under the influence of repulsive chemotaxis. In \cite{griette2023speed}, the authors provide an almost complete picture of the asymptotic dependence of the wave velocity on parameters representing the strength and length scale of chemotaxis. This latter work is based on the establishment of convergence to the Fisher-KPP traveling wave in porous media and a Fisher-KPP-Keller-Segel hyperbolic wave in certain asymptotic regimes. The proofs use a variety of techniques ranging from entropy methods and estimates of the decay of oscillations to a general description of the traveling wave. The main difference between the system we are studying here and the problem studied in \cite{griette2023speed}, lies in the description of the signal that generates the motion. In \cite{griette2023speed}, the signal is given by a repulsive Keller-Segel term.

Finally, we mention \cite{BiondoMarta2021TDoA,calvez2021mathematical,cochet2021hypoxia,demircigil2022,Demercigil}, which models the experiment in which a colony of Dictyostelium discoideum is able to escape hypoxia through a remarkable collective behavior. It is shown that oxygen consumption leads to self-generated oxygen gradients, which serve as directional cues and trigger a collective movement towards higher oxygen regions. In these works, the analysis is either purely numerical, or mathematical after the problem has been reduced to an ODE. The analysis we propose here is based on the study of the full PDE \eqref{eq:Clark_1} - \eqref{eq:Clark_2}, which has certain aspects in common with the hypotaxis model.

\section{Notations}\label{sec:notation}
We gather here some notation and some material that will be used in the sequel.
\begin{itemize}
	\item[$\bullet$] If $f:[0,+\infty[\times \mathbb{R}\longrightarrow \mathbb{R}$ is a real valued measurable function, for $1\leq p,q\leq +\infty$
	we will denote by $L^p_t(L^q_x)$ the $L^p$(in time)-$L^q$(in space) Lebesgue space which is given by the condition
	$$\|f\|_{L^p_t(L^q_x)}=\left(\int_{0}^{+\infty}\|f(t, \cdot)\|_{L^q}^p\D t\right)^{\frac1p}<+\infty,$$
	with the usual modifications when $p=+\infty$ or $q=+\infty$. 
	\item[$\bullet$] For $t>0$, the heat kernel is given by the function $h_t(x)=\frac{1}{\sqrt{4\pi t}}e^{-\frac{|x|^2}{4t}}$ for which we have the estimates
	\begin{equation}\label{HeatEstimate}
		\|\partial_x^kh_t\|_{L^p}\leq C t^{-\frac{k+(1-1/p)}{2}},
	\end{equation}
	with $k\in \mathbb{N}$ and $1\leq p\leq +\infty$. We also have that $\partial_t (h_t \ast f) = \Delta (h_t \ast f)$.
	\item[$\bullet$] The Sobolev spaces $W^{k,\infty}(\mathbb{R})$ are given as the set of measurable functions $f$ having weak derivatives of up to order
	$k$, with
	$$ \| f \|_{W^{1,\infty}(\mathbb{R})} = \| f \|_{L^\infty(\mathbb{R})} + \| \nabla^{k} f \|_{L^\infty(\mathbb{R})}<+\infty. $$
	By Morray's inequality, any such $f \in W^{k,\infty}$ is $C^{k-1,1}$.
	\item Some of our calculations will involve a precise constant, which requires a fixed definition of the Fourier transform. For any continuous
	$f \in L^1(\R)$, the Fourier transform $\hat f$ and its inverse are given by 
	\[
	\hat f (\xi)= \int_{\R} f(x) e^{-i x\xi} \dx \ \ \ \text{ and } \ \ \ f(x) = \frac{1}{2\pi} \int_{\R} \hat{f}(\xi) e^{i x \xi} \dxi.
	\]
	
\end{itemize}

\section{Stationary states and conditional traveling pulse solutions}\label{sec:stat}

Here we construct the special explicit solutions to \eqref{eq:reg} -- \eqref{def:v_c} that we will analyze in later sections. The first
of these is the stationary state, which always exists for all choices of parameters.

\subsection*{Proof of Theorem \ref{THM2} (iii)}
We assume that $x_0 = 0$. The pair $(\bar{S}_0(x), 0)$ will satisfy \eqref{eq:reg} -- \eqref{def:v_c} if and only if
\[
\bar{S}_0(x) - \beta^2 \partial^2_{xx} \bar{S}_0(x) = \alpha \gamma g_\epsilon (x) \ \ \ \text{ and } \ \ \ \partial_x \bar{S}_0(0) = 0 .
\]
The proof is then immediate from the Fourier transform of \eqref{eq:reg} in the $x$-variable. We see that $\bar{S}_0$ is given uniquely (in the class of integrable functions) by
\[
\hat{\bar{S}}_0(\xi) = \frac{\alpha \gamma \hat{g}_\epsilon(\xi)}{1+\beta^2 \xi^2} ,\]
so that
\[
\bar{S}_0(x) = \frac{\alpha\gamma}{2\pi} \int_{\R} \frac{\hat{g}_\epsilon(\xi) e^{ix \xi}}{1+\beta^2 \xi^2} \dxi 
= \frac{2 \alpha \gamma}{\pi} \int_0^\infty \frac{\hat{g}_\epsilon(\xi) \cos(x \xi)}{1+\beta^2 \xi^2} \dxi .
\]
The last equality used the fact that $g_\epsilon$ is an even function. Since $g_\epsilon$ is smooth, then so is $\bar{S}_0$. Furthermore, since $g_\epsilon$ is an even function, then likewise $\bar{S}_0(x)$ will be even, which means $\partial_x \bar{S}_0(0) = 0$.

If $x_0 \neq 0$ is given, then the pair $(\bar{S}_0(x-x_0), x_0)$ will solve  \eqref{eq:reg} -- \eqref{def:v_c}.\\
\qed

\subsection*{Proof of Theorem \ref{THM2} (i)}
For the (right-moving) traveling pulse, we seek a solution pair to \eqref{eq:reg} -- \eqref{def:v_c} of the form $(\bar{S}_0(x-v_c t), v_c t)$, for a fixed velocity $v_c > 0$. Of course,
by symmetry and translation invariance, $(\bar{S}_0(\pm(x- v_c t)- x_0), x_0 \pm v_c t )$ is also a solution corresponding to pulses traveling to the left and starting at an initial location $x_0$.
For simplicity we take $x_0 = 0$, while also mentioning that the stability analysis in the latter sections of this paper will need to consider $x_0$ arbitrary.

\begin{lemm}\label{LEM:Pulse}
	There exists $\eta^*_\epsilon > 0$ such that for all $\eta \ge \eta^*_\epsilon$, there exists a classical solution of \eqref{eq:reg} -- \eqref{def:v_c} of the form $S(t,x)= \bar{S}(x- v_c t)$ where $v_c >0$ is given implicitly by the relation
	\begin{equation}
		\frac{\eta}{\pi } \int_{0}^\infty \frac{\alpha ^2\gamma \hat g_\varepsilon (\xi)   \xi ^2}{(1 +\beta^2 \xi^2 )^2+ \alpha^2 v_c^2 \xi ^2 } \dxi=1.\label{eq:v_2}
	\end{equation}
	In the limit as $g_\epsilon$ approaches a delta function, we have that
	\begin{equation}
	\lim_{\epsilon \to 0+} \eta^*_\epsilon = \eta^*_0 := 4\alpha^{-2} \gamma^{-1} \beta^3 .
\label{eq:v_3}
	\end{equation}
\end{lemm}	
\begin{proof}
Under the traveling pulse ansatz (writing $w = x-v_c t$), \eqref{eq:reg} becomes
\[
\left( \frac{1}{\beta^2} - \frac{\alpha v_c}{\beta^2} \frac{\D}{\D w} - \frac{\D^2}{\D w^2} \right) \bar{S}(w)
= \frac{\alpha \gamma}{\beta^2} g_\epsilon (w) .
\]
Taking the Fourier transform, we obtain
\be\label{eq:reg_fourier}
\left(1 - \alpha v_c i \xi  +\beta^2 \xi ^2\right) \hat{\bar{S}}(\xi)= \alpha \gamma \hat g_\varepsilon (\xi),
\ee
hence
\begin{eqnarray}
\hat{\bar{S}}(\xi) &=& \frac{\alpha \gamma \hat g_\varepsilon (\xi) }{1 +\beta^2 \xi ^2- \alpha v_c i \xi  } \nonumber \\
&=& \frac{\alpha \gamma \hat g_\varepsilon (\xi) \left(1 +\beta^2 \xi ^2+ \alpha v_c i \xi \right) }{(1 +\beta^2\xi ^2)^2+ \alpha^2 v_c^2 \xi^2  },\label{eq:exp_S_Fourier}
\end{eqnarray}
and the traveling pulse solution must take the form
\[
\bar{S}(w) = \frac{2}{\pi} \int_0^\infty \frac{\alpha \gamma \hat{g}_\epsilon(\xi) \left((1+\beta^2 \xi^2) \cos(x \xi) - \alpha v_c \xi \sin(x \xi) \right)}
{(1+\beta^2 \xi^2)^2 + \alpha^2 v_c^2 \xi^2} \dxi .
\]
The above expression will always solve \eqref{eq:reg}, but to form a solution pair with $x_c(t) = v_c t$ it must also satisfy the condition \eqref{def:v_c},
which (after using \eqref{eq:exp_S_Fourier}) requires that
\begin{eqnarray}
v_c&=& - \eta \frac{d}{dw} \hat{\bar{S}}(w=0)\nonumber\\
&=& - \frac{\eta}{2\pi } \int_{\R} i \xi  \hat{\bar{S}}(\xi) \dxi\nonumber\\
&=&  v_c \frac{\eta}{\pi} \int_0^\infty \frac{\alpha^2 \gamma \hat{g}_\epsilon(\xi) \xi^2}{(1+\beta^2 \xi^2)^2 + \alpha^2 v_c^2 \xi^2} \dxi ,
\end{eqnarray}
which yields the implicit value of $v_c$ (given the parameters $\eta$, $\alpha$, $\beta$, $\gamma$, and $g_\epsilon$) seen in \eqref{eq:v_2},
provided that $v_c > 0$.

We now observe that the expression \eqref{eq:v_2} implies that $\eta$ must have a minimum value to permit the existence of traveling pulses. Taking the limit of the right hand side when $v_c \to 0$ yields
\[
 \frac{\eta}{\pi } \int_{0}^\infty \frac{\alpha ^2\gamma \hat g_\varepsilon (\xi)   \xi ^2 }{(1 +\beta^2 \xi ^2 )^2  } \dxi,
\]
while the limit of the right hand side when $v_c \to + \infty$ is $0$. Hence the question now is whether
\[
 \frac{\eta}{\pi } \int_{0}^\infty \frac{\alpha ^2\gamma \hat g_\varepsilon (\xi)    \xi ^2}{(1 +\beta^2 \xi ^2 )^2  } \dxi >1.
\]
We then have an exact expression for the critical value $\eta^*_\varepsilon$:
\[
\eta^*_\epsilon = \frac{\pi}{\alpha^2 \gamma} \left(  \int_{0}^\infty \frac{\hat g_\varepsilon (\xi)    \xi ^2 }{(1 +\beta^2 \xi ^2)^2  } \dxi  \right)^{-1} .
\]
We compute 
\begin{eqnarray*}
A&:=&  \int_{0}^\infty \frac{  \xi ^2 }{(1 +\beta^2 \xi ^2 )^2  } \dxi\\
&=& \frac{1}{\beta^2}  \int_{0}^\infty \frac{  1 }{1 +\beta^2 \xi ^2  } \dxi - \frac{1}{\beta^2} \int_{0}^\infty \frac{  1 }{(1 +\beta^2 \xi ^2 )^2  } \dxi\\
&=& \frac{\pi}{2\beta^3}-\frac{\pi}{4\beta^3} = \frac{\pi}{4\beta^3}.
\end{eqnarray*}
Note that, as $\epsilon \to 0$, $\hat{g}_\epsilon$ converges monotonically to the constant $1$ function (by choice of normalization in \eqref{def:gaussian}). Then
as $\epsilon \to 0$, $\eta^*_\epsilon$ is monotone decreasing and converges to
\[
\eta^*_0 = \frac{4\beta^3}{\pi}\frac{\pi}{\alpha ^2\gamma}=\frac{4\beta^3}{\alpha ^2\gamma} ,
\]
which proves \eqref{eq:v_3}.
\end{proof}

\section{Rescaled Equation}\label{sec:rescaled}

For the remainder of this work, we fix $\epsilon > 0$ and write only $g$ for the deformation term on the right-hand side, which we assume to be a fixed, smooth, even function with sufficient decay at infinity.\\
\\
From the mathematical perspective, it is simpler to examine equation \eqref{eq:reg} in rescaled coordinates to eliminate some of the physical constants.
This will help us to better see which parameters control the qualitative behavior of solutions.We start with the problem in its original form:
\begin{equation}\label{eq:lab}
\begin{split}
&\alpha \partial_t s(t,x) + s(t,x) - \beta^2 \partial_{xx}^2 s(t,x) = \alpha \gamma g(x-x_c(t)),\\
&\quad\quad \dot{x}_c(t) = -\eta \partial_x s(t,x_c(t)).
\end{split}
\end{equation}
Given a solution pair $(s, x_c)$ to \eqref{eq:lab}, a scaling factor $\lambda > 0$, and constants $a, b, c, d$ to be determined later, we then define
\[
s_\lambda (t,x) := \lambda^a s(\lambda^b t , \lambda^c x), \]
and
\[ x_{c\lambda}(t) := \lambda^d x_c(\lambda^b t).
\]
Then, provided that $d+c=0$, the pair $(s_\lambda, x_{c\lambda})$ solves
\[
\begin{split}
&\alpha_\lambda \partial_t s_\lambda(t,x) + s_\lambda(t,x) - \beta_\lambda^2 \partial_{xx}^2 s_\lambda(t,x) =
\alpha_\lambda \gamma_\lambda g_\lambda (x-x_{c\lambda}(t)) ,\\
&\quad\quad
\dot{x}_{c\lambda}(t) = -\eta \lambda^{b-a-2c} \partial_x s_\lambda(t,x_{c\lambda}(t)),
\end{split}
\]
where
\[
\alpha_\lambda = \alpha \lambda^{-b}, \ \ \beta_\lambda = \beta \lambda^{-c}, \ \ \gamma_\lambda = \gamma \lambda^{a+b},
\ \ \text{ and } \ \ g_\lambda(z) = g(\lambda^c z).
\]
This allows us to choose $\lambda$, $a$, $b$, and $c$ such that $\alpha_\lambda = \beta_\lambda = \gamma_\lambda = 1$. Up to
redefining $\eta$ and $g$, the model we study is essentially
\begin{equation}\label{eq:mat}
\begin{split}
\partial_t s(t,x) + s(t,x) - \partial_{xx}^2 s(t,x) &= g(x-x_c(t)) \\
\dot{x}_c(t) &= -\eta \partial_x s(t,x_c(t))
\end{split}
\end{equation}
Note how the only remaining parameter that determines the qualitative behavior of solutions to \eqref{eq:mat} is the (rescaled) stiffness $\eta$.
Under this formulation, the stationary solution $s(t,x) = \bar{s}_0(x)$ solves the equation
\begin{equation}\label{eq:stationary}
\bar{s}_0(x) - \partial_{xx}^2 \bar{s}_0(x) = g(x) ,
\end{equation}
and the traveling pulse solution $\tilde{s}(t,x) = \bar{s}(x-tv)$ solves the equation
\begin{equation}\label{eq:traveling}
\begin{split}
	& -v \bar{s}'(z) + \bar{s}(z) - \bar{s}''(z) = g(z), \\ 
	& v = -\eta \bar{s}'(0) ,
\end{split}	
\end{equation}
with explicit formula
\[
\bar{s}(z) = \frac{1}{2\pi} \int_{\R} \frac{e^{iz\xi} \hat{g}(\xi)}{1+\xi^2-iv\xi} \diff \! \xi .
\]
This formulation also makes it clear that the traveling pulse solutions only exist for sufficiently large stiffness $\eta$, as
\[
\bar{s}'(0) = -\frac v \eta = \frac{1}{2\pi} \int_{\R} \frac{i \xi \hat{g}(\xi)}{1+\xi^2-iv\xi} \diff \! \xi =
- \frac{1}{\pi}  \int_0^\infty \frac{v \xi^2 \hat{g}(\xi)}{(\xi^2+1)^2 + v^2 \xi^2} \diff \! \xi .
\]
We also remark that the model \eqref{eq:mat} tends to bring the total deformation (the integral of $s$ in space) towards an equilibrium given by $g$.
That is, defining
$$ s_\text{tot}(t) :=  \int_{\R} s(t,x) \diff \! x , $$
we see that
\[
\frac{\D}{\D t} s_\text{tot}(t) = \int_{\R} g(x) \diff \! x - s_\text{tot}(t) ,
\]
so that
\begin{equation}
\label{eq:deformation}
s_\text{tot}(t) = e^{-t} s_\text{tot}(0) + (1-e^{-t}) \int_{\R} g(x) \diff \! x .
\end{equation}

\section{Global Well-posedness}\label{sec:well}
This section demonstrates that the Cauchy problem for \eqref{eq:mat} is globally well-posed; i.e., that
solutions exist for all time, are unique, and are also smooth. The extent of this smoothness depends
on the regularity of $g$, while the uniformity in time of this smoothness depends additionally on $s_0$. The nonlinear nature of \eqref{eq:mat}
is quite mild, so these are all handled by a Picard iteration in a suitable $W^{k,\infty}$ space. Importantly, this argument requires
for $g$ to have some amount of smoothness.\\

\begin{prop}\label{prop:exists}
For any fixed integer $k \geq 2$, assume $s_0\in W^{k,\infty}(\mathbb{R})$ and $x_0 \in \R$ are given initial data. Assume also
that $g \in W^{k+1,\infty}$. Then the system
\eqref{eq:mat} admits a unique solution pair $(s,x_c)$  in $L^\infty_{[0, T]} W^{k,\infty}_{\mathbb{R}} \times L^\infty_{[0,T]}$,
with any $T>0$, starting from that initial data.
\end{prop}

\begin{proof}
Suppose that $s_0$ and $x_0$ are given initial data. We first define the approximate solutions by setting
\[
s^0(t,x) = s_0(x) \quad  \textrm{ and } \quad x^0_c(t) = x_0 ,
\]
and for each $n$ solving
\begin{equation}\label{eq:approx}
\begin{split}
&\partial_t s^n(t,x) + s^n(t,x) - \partial_{xx}^2 s^n(t,x) = g(x-x^n_c(t)) \\
&\quad\quad \dot{x}^n_c(t) = -\eta \partial_x s^{n-1}(t,x^{n-1}_c(t)).
\end{split}
\end{equation}
The system \eqref{eq:approx} solves for $x^n_c$ by direct integration and then solves for $s^n$ by the heat equation with
damping, so $(s^n, x^n_c)$ exists and is smooth for all time. Recalling \eqref{HeatEstimate}), we have
\begin{equation}\label{eq:approx_sol}
\begin{split}
s^n(t,x) &= e^{-t} h_t\ast s_{0}(x) + \int_0^t  e^{-(t-\tau)} h_{t-\tau}\ast  g(x-x^{n-1}_c(\tau)) \diff \! \tau, \\
(\partial_x)^k s^n(t,x) &= e^{-t} h_t\ast s_{0}^{(k)}(x) + \int_0^t e^{-(t-\tau)} h_{t-\tau}\ast  g^{(k)}(x-x^{n-1}_c(\tau)) \diff \! \tau, \\
x^n_c(t) &= x_0 - \eta \int_0^T s_x^{n-1}(\tau, x^{n-1}_c(\tau)) \diff \! \tau.
\end{split}
\end{equation}
More importantly, the approximate solution $s^n$ is bounded uniformly in every $W^{k,\infty}$ space. Indeed,
since convolution with $h_t$ is bounded by 1 as an operator from $L^\infty$ to $L^\infty$,
\begin{equation}\label{eq:Picard-ball}
\| s^n \|_{W^{k,\infty}_x}(t) \leq e^{-t} \| s_0 \|_{W^{k,\infty}} + (1-e^{-t}) \| g \|_{W^{k,\infty}} \leq M_k,
\end{equation}
where $M_k$ depends on $s_0$ and $g$, but not on $t$.\\
\\
For a fixed $T>0$ to be determined later, our Banach space will be $L^\infty_{[0,T]} W^{k,\infty}_\R$. Observe
that \eqref{eq:Picard-ball} demonstrates that the sequence $\lbrace s^n \rbrace_n$ is bounded. Define
\[
\Psi_k^n(t) = \sup_{(\tau, x) \in [0, t] \times \R} \left|(\partial_x)^k s^n(\tau, x) - (\partial_x)^k s^{n-1}(\tau,x) \right| ,
\] 
and
\[ 
\sigma^n(t) = \sup_{0 \leq \tau \leq t} \left| x^n_c(\tau) - x^{n-1}_c(\tau) \right| .
\]
Using \eqref{eq:approx_sol}, the usual estimate for quadratic nonlinearities, and the mean value theorem, we have
\[
\begin{split}
\sigma^n(T) &\leq \eta \int_0^T \left| s_x^{n-1}(\tau,x^{n-1}_c(\tau)) - s_x^{n-2}(\tau,x^{n-2}_c(\tau)) \right| \diff \! \tau \\
&\leq \eta \int_0^T \left| s_x^{n-1}(\tau, x^{n-1}_c) - s_x^{n-1}(\tau, x^{n-2}_c) \right|
	+ \left| s_x^{n-1}(\tau, x^{n-2}_c) - s_x^{n-2}(\tau, x^{n-2}_c) \right| \diff \! \tau \\
&\leq \eta \int_0^T \| s^{n-1} \|_{W^{2,\infty}}(\tau) \sigma^{n-1}(\tau) \diff \! \tau + \eta \int_0^T \Psi_1^{n-1}(\tau) \diff \! \tau \\
&\leq \eta T M_2 \sigma^{n-1}(T) + \eta T \Psi_1^{n-1}(T).
\end{split}
\]
From the Duhamel formula in \eqref{eq:approx_sol} and again the mean value theorem we also have
\[
\begin{split}
\Psi^n_k(T) &\leq \int_0^T e^{-(T-\tau)} \left| h_{t-\tau} \ast g^{(k)}(x-x^n_c(\tau)) - h_{t-\tau} \ast g^{(k)}(x-x^{n-1}_c(\tau)) \right| \diff \! \tau \\
&\leq \int_0^T \| h_{t-\tau}\ast g^{(k+1)} \|_{L^\infty} \left| x^n_c(\tau) - x^{n-1}_c(\tau) \right| \diff \! \tau \leq T \sigma^n (T) \| g \|_{W^{k+1,\infty}} \\
&\leq \eta T^2 M_{k+1} \left( M_2 \sigma^{n-1}(T) + \Psi_1^{n-1}(T) \right).
\end{split}
\]
Hence, the quantity $\Phi^n(T) := \sigma^n(T) + \Psi_0^n(T) + \Psi_1^n(T) + \cdots + \Psi_k^n(T)$ (for $k \geq 2$) will satisfy
\[
\Phi^n(T) \leq C_k T^2 \Phi^{n-1}(T),
\]
where $C_k$ depends on $\eta$ and $M_0 , \cdots , M_{k+1}$. For $T$ sufficiently small, we conclude that the sequence $\lbrace s^n \rbrace_n$
converges in our Banach space to a unique solution $s$ of \eqref{eq:mat}, and that the sequence $\lbrace x^n_c \rbrace_n$ also
converges to $x_c$ in $L^\infty$.\\
\\
To extend to a global solution, we reiterate the problem starting from $s(T,x)$. Importantly, recall that \eqref{eq:Picard-ball} holds true
a priori for all time and all $n$, so that the value of the $M_k$ used in the reiterated argument do not increase with each reiteration. That is, the $M_k$
used to define $C_k$ will still only depend on $s_0(x)$, and not on $s(T,x)$. Thus each reiteration
extends the solution for a time interval of the same length, which gives the global existence.\\
\\
From \eqref{eq:approx_sol} and the convergence of $(s^n,x^n_c)$ to $(s,x_c)$ in $L^\infty_t W^{k\infty}_x \times L^\infty$, we see that
$x_c$ is in fact differentiable, and that (along with \eqref{eq:Picard-ball})
\begin{equation}\label{eq:strong}
\begin{split}
\dot{x}_c(t) &= -\eta s_x (t,x_c(t)), \ \ \ x_c(t) = x_0 - \eta \int_0^t s_x(\tau, x_c(\tau)) \diff \! \tau \\
s(t,x) &= e^{-t} (h_t \ast s_0)(x) + \int_0^t e^{-(t-\tau)} (h_{t-\tau} \ast g)(x-x_c(\tau)) \diff \! \tau \\
\| s \|_{W^{k,\infty}_x}(t) &\leq M_k.
\end{split}
\end{equation}
To prove uniqueness, we observe that any strong solution must also satisfy \eqref{eq:strong}.
We then assume that we have two such solutions $(s_1, x_{c,1})$ and $(s_2, x_{c,2})$ that both start from initial data $(s_0, x_0)$.
Define
\[
F(t) := \left| x_{c,1}(t) - x_{c,2}(t) \right| + \sup_x \left| \partial_x s_1(t,x) - \partial_x s_2(t,x) \right|,
\]
observing that $F(0)=0$. Writing $G_t(x) = (h_t \ast g)(x)$, we use \eqref{eq:strong} and the mean value theorem to obtain
\[
\begin{split}
\left| \partial_x s_1(t,x) - \partial_x s_2(t,x) \right|
&\leq \int_0^t e^{-(t-\tau)} \left| G'_{t-\tau}(x-x_{c,1}(\tau)) - G'_{t-\tau}(x-x_{c,2}(\tau)) \right| \diff \! \tau \\
&\leq \int_0^t 1 \cdot \left| G''_{t-\tau}(\theta) \right| \left| x_{c,1}(\tau) - x_{c,2}(\tau) \right| \diff \! \tau \\
&\leq \int_0^t \| g \|_{W^{2,\infty}} F(\tau) \diff \! \tau
\end{split}
\]
and
\[
\begin{split}
\left| x_{c,1}(t) - x_{c,2}(t) \right| &\leq \eta \int_0^t \left| \partial_x s_1(\tau,x_{c,1}(\tau)) - \partial_x s_2(\tau,x_{c,2}(\tau)) \right| \diff \! \tau \\
&\leq \eta \int_0^t \left| \partial_x^2 s_1(\tau,\theta) \right| \left| x_{c,1} - x_{c,2} \right| +
\left| \partial_x s_1(\tau, x_{c,2}) - \partial_x s_2(\tau x_{c,2}) \right| \diff \! \tau \\
&\leq \eta \int_0^t (M_2 + 1) F(\tau) \diff \! \tau.
\end{split}
\]
Hence, for some constant $C$,
\[
F(t) \leq \int_0^t C F(\tau) \diff \! \tau,
\]
which allows us to conclude that $F(t) \equiv 0$ by the integral form of Gr\"onwall's lemma. This establishes uniqueness.\\
\end{proof}

\begin{prop}\label{prop:regular}
The solution pair $(s(t,x), x_c(t))$ constructed in Proposition \ref{prop:exists} is in fact a \emph{strong solution}, meaning that
for every integer $0 \leq m \leq \lfloor k/2 \rfloor$, we have $\partial_t^m s \in L^\infty_t W^{k-2m,\infty}_x$, and that $x_c \in W^{\lceil k/2 \rceil,\infty}(\mathbb{R}_+)$, with bounds that are uniform in time.
\end{prop}

\begin{proof}
We differentiate in $t$ the expression for $s(t,x)$ from \eqref{eq:strong}, which gives
\[
\partial_t s = e^{-t} (\Delta-1) (h_t \ast s_0) + g(x-x_c(t)) + \int_0^t e^{-(t-\tau)} (\Delta - 1) (h_{t-\tau} \ast g)(x-x_c(\tau)) \diff \! \tau.
\]
Since $s_0 \in W^{k,\infty}$ and $g \in W^{k+1,\infty}$, it follows that $\partial_t s \in L^\infty_t W^{k-2,\infty}_x$.\\
\\
The essence of the proof is that we repeat this process to take higher order time derivatives, but the expressions become more complicated
once factors of $\dot{x}_c$ start to appear. But each such factor is (by \eqref{eq:strong})
replaced with $-\eta s_x(t,x_c(t))$, resulting in terms that are lower-order and bounded in $L^\infty_x$. To be precise, we must observe that
\[
\partial_t \left( (\partial_t^l \partial_x^{l'} s)(t,x_c(t)) \right) = (\partial_t^{l+1} \partial_x^{l'} s)(t,x_c(t))
-\eta  \partial_x s(t,x_c(t)) (\partial_t^l \partial_x^{l'+1} s)(t,x_c(t)).
\]
The proof then proceeds by induction on $m$, since
\[
\partial_t^m s = e^{-t} (\Delta-1)^m(h_t \ast s_0) + \sum_{\substack{i+j+l =m-1 \\ j \geq 1}} \phi_{i,j,l}(t,x) +
\int_0^t e^{-(t-\tau)} (\Delta-1)^m (h_{t-\tau} \ast g)(x-x_c(\tau)) \diff \! \tau,
\]
where
\[
\phi_{i,j,l} = c_{i,j,l} (\Delta-1)^i g^{(j)}(x-x_c(t)) \left( \sum_{l'} \psi_{l',l} \right).
\]
Here, each term $\psi_{l',l}$ is a product with at most $j+l$ factors of the form $(\partial_t^\alpha \partial_x^\beta s)(t,x_c(t))$.
For each of these factors, $\alpha + \beta \leq l \leq m-2$.\\
\\
The first term is clearly in $L^\infty_x$ (uniformly in $t$) because $s_0 \in W^{k,\infty}$ and we assume that $2m \leq k$. For the third term,
$(\Delta-1)^m (h_{t-\tau} \ast g)$ is also bounded in $L^\infty_x$ uniformly in time, and the integral in $\tau$ remains bounded in time due
to the exponential factor $e^{-(t-\tau)}$.\\
\\
For the complicated middle term, our inductive hypothesis states that $\partial_t^{m'} \partial_x^{k-2m'} s \in L^\infty_x$,
for every $m' < m$. Taking $m' = \alpha$, every factor comprising $\psi_{l',l}$ will belong to $L^\infty_x$ provided that $\beta \leq k-2\alpha$.
But $\beta \leq m-2-\alpha \leq k-2\alpha$ because $\alpha < m \leq \lfloor k/2 \rfloor$. This proves the regularity claimed for $s$,
up to $m = \lfloor k/2 \rfloor$.\\
\\
For the regularity of the cell-cluster's trajectory, we differentiate the expression for $x_c$ in \eqref{eq:strong} repeatedly to get
\[
\frac{\D^m}{\D t^m} x_c(t) = \sum_{l'} \tilde{\psi}_{l'},
\]
where each term $\tilde{\psi}_{l'}$ is similarly a product of at most $m$ factors of the form $(\partial_t^\alpha \partial_x^\beta s)(t,x_c(t))$.
In fact, these terms are precisely of the same form as $\psi_{l',m}$, which one would encounter in the expression for $\partial_t^{m+2} s$,
so that $\alpha + \beta \leq m$. However, the first time derivative on $s$ does not appear until $m=2$, so that we also have
$\alpha \leq m-1$. Then every factor comprising $\tilde{\psi}_{l'}$ will belong to $L^\infty_x$ provided that $\beta \leq k-2\alpha$,
which still holds because $\beta \leq m-\alpha \leq k-2\alpha$ since $\alpha < m \leq \lfloor k/2 \rfloor$. Lastly, this bound will
hold for all $m$ up to a maximum of $\lceil k/2 \rceil$, which proves the regularity claimed for $x_c$.\\
\end{proof}

The last part of Theorem \ref{THM1} follows from the formulas for $s$ and $x_c$ in \eqref{eq:strong}. If one does not care about uniformity in time,
then $h_t \ast s_0$ is already $C^\infty$ for all $t>0$. While the same is true for $h_{t-\tau} \ast g(x-x_c(\tau))$, the integration in $\tau$
prevents that quantity from being automatically $C^\infty$. Indeed, according to \eqref{HeatEstimate}, $\| h_{t-\tau} \ast g \|_{W^{k,\infty}}$
will fail to be integrable as $\tau$ approaches $t$. But if we assume that $g \in C^\infty$, we can indeed conclude that $s \in C^\infty$.
From there, likewise $x_c \in C^\infty$.

\section{Stability of the Traveling Pulse}\label{sec:stab_tw}

Here we investigate small perturbations of the traveling pulse $(\bar{s}(x-vt-x_0), vt+x_0)$. We assume that $s_0(x) = \bar{s}(x) + z_0(x)$ and that
$x_c(0) = 0$, centering the problem to simplify the notation. We observe that if the pair $(z,y)$ solves the system
\begin{equation}\label{eq:perturb}
\begin{split}
&\partial_t z(t,x) + z(t,x) - \partial_{xx}^2 z(t,x) = -\dot{y}(t) \bar{s}'(x-vt+y(t)) \\
&\quad\quad \dot{y}(t) = \eta \partial_x z(t,vt-y(t))
\end{split}
\end{equation}
then the pair $(s(t,x), x_c(t)) := \left( \bar{s}(x-vt+y(t)) + z(t,x) , vt - y(t) \right)$ will solve \eqref{eq:mat}. That is, we have found
the system for the perturbations around the traveling pulse. To establish stability, we must prove that these perturbations decay to zero;
note that, since a pulse can be translated in the $x$-direction, we expect $z$ and $\dot{y}$ to converge to zero, whereas $y$ will
converge to some constant. To this end, we prove the following.

\begin{prop}\label{prop:stable_pulse}
Let $(z(t,x), y(t))$ solve the system \eqref{eq:perturb}, and assume that $g$ satisfies the following assumption of
Gaussian decay up to second order:
\begin{equation}\label{eq:g-decay}
\max \left( |g(z)|, |g'(z)|, |g''(z)| \right) \leq M e^{-mz^2}, \text{ for some } M, \ m > 0  \text{ with } m > \bar{C} M^2 \eta^2 / |v|^2.
\end{equation}
Then there exists $\epsilon_0 > 0$ depending on $M$, $m$, $\bar{C}$, and $\eta$ such that, whenever the initial perturbation satisfies
\[
\max \left(\| z_0 \|_{L^\infty}, \| z_0' \|_{L^\infty}, |y_0| \right) < \epsilon_0,
\]
then for some constants $C',\delta > 0$ and all $t$ we have $\max \left( \| z \|_{L^\infty_x}(t), |\dot{y}(t)| \right) < C' e^{-\delta t}$.
\end{prop}

We remark that the constant $\bar{C}$ appearing in \eqref{eq:g-decay} does not depend on the initial data or any of the parameters.
The assumptions on $g$ help to control the accumulation of long-range interactions. As will be seen in the proof, the constant $\delta$
can be taken arbitrarily close to $1$, at the expense of making $\epsilon_0$ smaller or $C'$ larger.
We comment that the limiting rate of decay $e^{-t}$ is expected from the "$z_t + z$" appearing in the equation.

\begin{proof}
We remark that, while $h_t$ is defined for positive $t$ as a function, $h_t$ also makes sense as a Fourier multiplier operator
for $t \leq 0$. The proof will require some differential identities involving this operator, which by abuse of notation we will continue to
denote as convolution with $h_t$; i.e., for all $t$
\[
(h_t \ast F)(x) := \frac{1}{2\pi} \int_{\mathbb{R}} e^{-|\xi|^2 t} \hat{F}(\xi) e^{i x \xi} \diff \! \xi.
\]
Our first goal is to prove exponential decay for $|\dot{y}(t)|$. The Duhamel formula obtained directly from \eqref{eq:perturb} would yield that
\[
z(t,x) = e^{-t} (h_t \ast z_0)(x) - \int_0^t \dot{y}(\tau) e^{-(t-\tau)} (h_{t-\tau} \ast \bar{s}')(x-v\tau+y(\tau)) \diff \! \tau .
\]
Unfortunately, it is more difficult to examine the $\bar{s}$ in the integrand.
Instead, we observe that the right hand side of \eqref{eq:perturb} is in fact equal to
\[
-\partial_t \left( \bar{s}(x-vt+y(t)) \right) - v\bar{s}'(x-vt+y(t)) .
\]
Therefore
\[
\begin{split}
e^{\tau} (h_{-\tau} \ast \bar{s}')(x-v\tau+y(\tau)) &= - \left(h_{-\tau} \ast (\partial_\tau e^\tau \bar{s}) \right)(x-v\tau+y) \\
&\quad\quad\quad -v e^{\tau} \left( h_{-\tau} \ast \bar{s}' \right)(x-v\tau+y) +e^\tau \bar{s}(x-v\tau+y) .
\end{split}
\]
As can be verified with a Fourier transform, we have that
\[
h_{-\tau} \ast \partial_\tau (F(\tau,x)) = \partial_\tau \left( h_{-\tau} \ast F(\tau,x) \right) + h_{-\tau} \ast F_{xx}(\tau,x) .
\]
Recalling also \eqref{eq:traveling}, the new Duhamel formula reads
\begin{equation*}
\begin{split}
z(t,x) &= e^{-t} (h_t \ast z_0)(x) - \int_0^t e^{-t} \left( h_{t-\tau} \ast (\partial_\tau e^\tau \bar{s}) \right)(x-v\tau+y(\tau)) \diff \! \tau \\
&\quad\quad\quad + \int_0^t e^{- (t-\tau)} \left( -v (h_{t-\tau} \ast \bar{s}')(x-v\tau+y(\tau)) + (h_{t-\tau} \ast \bar{s})(x-v\tau+y(\tau)) \right) \diff \! \tau \\
&= e^{-t} (h_t \ast z_0)(x) -e^{-t} \left( h_t \ast \int_0^t \partial_\tau \left( e^{\tau} (h_{-\tau} \ast \bar{s})(\cdot-v\tau+y(\tau)) \right) \diff \! \tau \right)(x) \\
&\quad\quad\quad + \int_0^t e^{- (t-\tau)} \left. \left( h_{t-\tau} \ast (\bar{s} - v \bar{s}' - \bar{s}'') \right) \right|_{x-v\tau+y(\tau)} \diff \! \tau \\
&= e^{- t} (h_t \ast z_0)(x) - \bar{s}(x-vt+y(t)) - e^{- t} (h_t \ast \bar{s})(x)\\
&\quad\quad\quad + \int_0^t e^{- (t-\tau)} (h_{t-\tau} \ast g)(x-v\tau+y(\tau)) \diff \! \tau 
\end{split}
\end{equation*}
By \eqref{eq:perturb} and a slight change of variables, we then have
\[
\begin{split}
\frac{\dot{y}(t)}{\eta} &= (h_t \ast z_0')(vt-y(t)) - \bar{s}'(0) + e^{- t} (h_t \ast \bar{s})(vt-y(t)) \\
& \quad\quad\quad +\int_0^t e^{- \tau} (h_\tau \ast g')(v\tau + y(t-\tau)-y(t)) \diff \! \tau .
\end{split}
\]
Note that $-\bar{s}'(0) = v / \eta$, indicating that the above expression, though correct, still fails to capture the perturbative nature of $y$. The
insight that produces a better expression is to fix a $\tilde{y} \in \R$ and examine the Duhamel formula for the function $\tilde{s}(t,x) :=
\bar{s}(x-vt+\tilde{y})$. This yields
\[
\tilde{s}'(t,x) = e^{-t} (h_t \ast \tilde{s}')(0,x) + \int_0^t e^{- (t-\tau)} (h_{t-\tau} \ast g')(x-v\tau+\tilde{y}) \diff \! \tau .
\]
Fixing a $T>0$ and setting $\tilde{y} = y(T)$, we therefore have that
\[
\bar{s}'(x-vt+y(T)) = e^{- t} (h_t \ast \bar{s}')(x-y(T)) + \int_0^t e^{- (t-\tau)} (h_{t-\tau} \ast g')(x-v\tau+y(T)) \diff \! \tau.
\]
Note the different places where $t$ or $T$ appear. The expression would be false if $T$ were not fixed, but instead replaced with $t$. This
gives us a refined Duhamel expression for $\partial_x z$ at time $T$:
\[
\begin{split}
\partial_x z(T,x) &= e^{- T} (h_T \ast z_0')(x) + e^{- T} \left( (h_T \ast \bar{s}')(x)-(h_T \ast \bar{s}')(x-y(T)) \right) \\
&\quad\quad+ \int_0^T e^{- (T-\tau)} \left( (h_{T-\tau} \ast g')(x-v\tau+y(\tau)) - (h_{T-\tau} \ast g')(x-v\tau+y(T)) \right) \diff \! \tau ,
\end{split}
\]
and likewise for $\dot{y}(T)$:
\begin{equation}\label{eq:dot_y_useful}
\begin{split}
\dot{y}(T) &= \eta e^{- T} \left( (h_T \ast z'_0)(Tv-y(T)) + (h_T \ast \bar{s}')(Tv-y(T)) - (h_T \ast \bar{s}')(Tv) \right) \\
&\quad\quad+ \eta \int_0^T e^{- \tau} \left( (h_\tau \ast g')(v\tau + y(T-\tau) - y(T)) - (h_\tau \ast g')(v\tau) \right) \diff \! \tau .
\end{split}
\end{equation}
To proceed, we observe that there exists some $C_1 > 0$ depending only on $z_0$, $\eta$ and $g$ such that
\begin{equation}\label{eq:dot_y_first_part}
\left| \eta e^{- T} \left( (h_T \ast z'_0)(Tv-y(T)) + (h_T \ast \bar{s}')(Tv-y(T)) - (h_T \ast \bar{s}')(Tv) \right) \right| \leq C_1 e^{-T}.
\end{equation}
The dependence of $C_1$ on $g$ is implicit through the dependence on $\bar{s}$.\\
\\
Now, to prove the exponential decay of $\dot{y}$, we first assume that there exists a $t_1 > 0$ and a constant $a \in (0,1)$ such that
\begin{equation}\label{eq:y-induct}
|\dot{y}(t)| < a e^{-t} \ \ \text{ for all } \ \ t \in [0,t_1) .
\end{equation}
Choosing $\epsilon_0$ small enough (in terms of $a$) guarantees that this holds at time $t=0$ (using, say, \eqref{eq:dot_y_useful}).
By smoothness of solutions, such a positive $t_1$ always exists, though any lower bound must depend on the size of the initial data $(z_0, z_0')$.
We assume that $t_1$ is in fact the latest time where \eqref{eq:y-induct} holds. That is, $t_1$ is a crossing point where $|\dot{y}(t_1)| = a e^{-t}$.
We then let $T=t_1$ in \eqref{eq:dot_y_useful}. By the mean value theorem and \eqref{eq:dot_y_first_part}, we have
\begin{equation}\label{eq:dot_y_total}
|\dot{y}(T)| \leq C_1 e^{-T} + \eta \int_0^T e^{-\tau}\left| (h_\tau \ast g'')(\theta_\tau) \right| \left|y(T-\tau) - y(T) \right| \diff \! \tau
\end{equation}
where $\theta_\tau$ is some point between $v\tau + y(T-\tau) - y(T)$ and $v\tau$. Note that, by \eqref{eq:y-induct},
$|y(T-\tau) - y(T)| \leq ae^{-T} (e^\tau - 1)$. Choosing $a$ small enough, we have that this is smaller than $|v|/2$. Hence, $|\theta_\tau| \geq v\tau / 2$.\\
\\
Moreover, given the Gaussian bounds on $g''$ from \eqref{eq:g-decay}, we also obtain pointwise Gaussian decay for $h_\tau \ast g''$.
Indeed, since we assume that $\tau >0$,
\[
\begin{split}
\left| h_\tau \ast g''(x) \right| &\leq \int \frac{1}{\sqrt{4\pi \tau}} e^{-\frac{|y-x|^2}{4\tau}} \left| g''(y) \right| \diff \! y \\
&\leq \frac{M}{\sqrt{4\pi \tau}} \int \exp\left(-\frac{|y-x|^2 + 4\tau m x^2}{4\tau} \right) \diff \! y \\
&= \frac{M}{\sqrt{4\pi \tau}} \int \exp\left( -\frac{|y-x/(1+4\tau m)|^2}{4\tau / \sqrt{1+4\tau m}}-\frac{mx^2}{1+4\tau m} \right) \diff \! y\\
&= \frac{M}{(1+4\tau m)^{1/4}} \exp\left(-\frac{m x^2}{1+4\tau m} \right) \leq M e^{-mx^2} .
\end{split}
\]
Therefore $|(h_\tau \ast g'')(\theta_\tau)| \leq M \exp(-m|v|^2\tau^2 / 4)$, and \eqref{eq:dot_y_total} becomes
\[
\begin{split}
|\dot{y}(T)| &\leq C_1 e^{-T} + a M \eta e^{-T} \int_0^T e^{-\frac{m|v|^2 \tau^2}{4}} \left( 1 - e^{-\tau} \right) \diff \! \tau \\
&\leq e^{-T} \left( C_1 + aM\eta \int_0^\infty e^{-\frac{m|v|^2 \tau^2}{4}} \diff \! \tau \right)
\leq e^{-T} \left( C_1 + a \frac{ M \eta C_2}{|v| \sqrt{m}} \right) ,
\end{split}
\]
for some fixed $C_2 > 0$ which does not depend on $g$, the initial data, or any of the parameters.
If we have $\bar{C} > 9 C_2^2$ in \eqref{eq:g-decay}, and that the initial data are
sufficiently small to make $C_1 < a/3$, then we conclude that
\[
|\dot{y}(t_1)| \leq \frac 2 3 a e^{-t_1} ,
\]
which contradicts our crossing-time assumption. Thus, \eqref{eq:y-induct} holds for all $t > 0$.\\
\\
The exponential decay of $z$ is then almost immediate from \eqref{eq:perturb} and a barrier argument. 
For fixed $\delta \in (0,1)$ and any $\epsilon > 0$, define
\[
\phi^\epsilon(t,x) = C' e^{-\delta t} + \epsilon (x^2+2).
\]
If we assume that $\epsilon_0 < C'$, then $z_0(x) < \phi^\epsilon(0,x)$. We claim that $z(t,x) < \phi^\epsilon(t,x)$ for all $t$, $x$,
and $\epsilon$. Since $\epsilon>0$ was arbitrary, and since our argument will hold equally well for $-z(t,x)$, this would imply that
$|z(t,x)| < C' e^{-\delta t}$, which (along with \eqref{eq:y-induct}) completes the proof.\\
\\
To prove that claim, we proceed by contradiction. From Proposition \ref{prop:exists}, we know that solutions
to \eqref{eq:mat} are bounded in $x$ uniformly. However $\phi^\epsilon$ grows to $+\infty$ as $|x| \rightarrow \infty$.
Hence, if the claim were false, there would have to exist a first time $t_1 >0$ and a crossing point $x_1$; i.e., $z(t,x) < \phi^\epsilon(t,x)$
for all $t < t_1$, $z(t_1, x_1) = \phi^\epsilon(t_1,x_1)$, and $z(t_1, x) \leq \phi^\epsilon(t_1, x)$.
If we define $w(t,x) := \phi^\epsilon(t,x) - z(t,x)$, then
\begin{equation}\label{eq:w-crossing}
\partial_t w(t_1,x_1) \leq 0 \ \text{ and } \ \partial_{xx}^2 w(t_1,x_1) \geq 0.
\end{equation}
Using \eqref{eq:perturb}, we see that $w$ satisfies the equation
\[
w_t + \delta w - w_{xx} = \epsilon \delta x^2 + (1-\delta) z + \dot{y}(t) \bar{s}'(x-vt+y(t)).
\]
If we evaluate at $(t_1,x_1)$ and use \eqref{eq:w-crossing}, we get
\[
\begin{split}
0 &\geq w_t(t_1,x_1) \geq w_t(t_1,x_1) - w_{xx}(t_1, x_1)
\geq \epsilon \delta x_1^2 + (1-\delta)\phi^\epsilon(t_1,x_1) - |\dot{y}(t_1)| \| \bar{s}' \|_{L^\infty} \\
& \geq C' (1-\delta) e^{-\delta t_1} - a \| \bar{s}' \|_{L^\infty} e^{-t_1} > 0 ,
\end{split}
\]
provided that $a \| \bar{s} \|_{W^{1,\infty}} < C' (1-\delta)$. But this is clearly a contradiction, so no such
crossing point could have existed, proving the claim.\\
\end{proof}

Taking $k$ derivatives in $x$ of \eqref{eq:perturb}, we see that $\partial_x^k z$ satisfies the nearly identical equation
\[
\partial_t (\partial_x^k z) + \partial_x^k z - \partial_x^2 (\partial_x^k z) = -\dot{y}(t) \bar{s}^{(k+1)}(x-vt+y(t)) .
\]
We can then repeat the last part of the proof of Proposition \ref{prop:stable_pulse} with the same barrier to conclude exponential
decay of higher-order $x$-derivatives of $z$, so long as we can guarantee a uniform bound-in-time bound on $\| z \|_{W^{k,\infty}}$.
This can be accomplished by assuming the necessary regularity of $z_0$ and $g$, using Proposition \ref{prop:exists}. We summarize this
with the following.

\begin{coro}\label{cor:higher_decay}
With the same assumptions as Proposition \ref{prop:stable_pulse}, further assume that $z_0 \in W^{k,\infty}$ and $g \in W^{k+1,\infty}$
(which in particular means that $\bar{s} \in W^{k+2,\infty}$). Then, with the same $\epsilon_0$ but a possibly larger constant $C'$, we have
\[
\max \left( \| z \|_{L_x^\infty}(t) , \| z \|_{W^{k,\infty}_x}(t), |\dot{y}(t)| \right) \leq C' e^{-\delta t} .
\]
\end{coro}

\section{Stability of the Stationary State}\label{sec:stab_stat}
Here we analyze the stationary solution $\bar{s}_0$ from \eqref{eq:stationary}. Since a traveling pulse is only guaranteed for
sufficiently large stiffness $\eta$, we expect that the stationary state should be stable in the case of low stiffness. More than this,
we are able to prove that the family of stationary states created by translations of the profile $\bar{s}_0$ is exponentially attracting for
the dynamics of \eqref{eq:mat}.

\begin{prop}\label{prop:stable_stat}
Let $(s(t,x),x_c(t))$ solve the system \eqref{eq:mat} with initial data $(s_0(x), x_0)$, $s_0 \in W^{k,\infty}$, $g \in W^{k+1,\infty}$, $k \geq 2$.
Then, for any $\delta \in (0,1)$, there exists $C_1$, $\eta_0$, and $\bar{x}$ such that for all $\eta < \eta_0$ we have
\begin{equation}\label{eq:stable_stat}
\| s(t,x) - \bar{s}_0(x-\bar{x}) \|_{W^{1,\infty}_x}(t) \leq C_1 e^{-\delta t} .
\end{equation}
\end{prop}
The constant $C_1$ above depends on the initial data $s_0$ and on $\delta$, but the threshold $\eta_0$ only depends on $g$ and $\delta$
(i.e., the exponential convergence to the stationary state is unconditional for all sufficiently "loose" collagen networks.
This first step requires a slightly more careful barrier argument than Proposition \ref{prop:stable_pulse}. The proof of part (iv) of Theorem
\ref{THM2} will the follow from the corollary below.
\begin{proof}
At first glance, we treat the problem perturbatively, and assume that $s_0(x) = \bar{s}_0(x) + z_0(x)$ and that $x_c(0) = x_0 = 0$.
Through a calculation similar to the one that found \eqref{eq:perturb}, we see that if the pair $(z(t,x),y(t))$ solves the system
\begin{equation}\label{eq:perturb_stat}
\begin{split}
&\partial_t z(t,x) + z(t,x) - \partial_{xx}^2 z(t,x) = -\dot{y}(t) \bar{s}_0'(x+y(t)) \\
&\quad\quad \dot{y}(t) = \eta \partial_x z(t,-y(t))
\end{split}
\end{equation}
then the pair $(s(t,x), x_c(t)) := \left( \bar{s}_0(x+y(t)) + z(t,x) , - y(t) \right)$ will solve \eqref{eq:mat}. Moreover, if $z_0 \in W^{k,\infty}$ and
$g \in W^{k+1,\infty}$, Proposition \ref{prop:exists} guarantees that $\| z \|_{W^{m,\infty}_x}(t) \leq M_m$ for all $m \leq k$
(and implies that $\| \bar{s}_0 \|_{W^{m\infty}} \leq \bar{M}_m$ for all $m \leq k+3$). In particular we have
\begin{equation}\label{eq:max_dist}
|y(t)| \leq \eta M_1 t .
\end{equation}
If we write $w(t,x) = \partial_x z(t,x)$, then the near-linearity of the PDE gives
\[
\partial_t w(t,x) + w(t,x) - \partial_{xx}^2 w(t,x) = -\dot{y}(t) \bar{s}_0''(x+y(t)).
\]
We proceed with a barrier argument, but different from the one in the previous section. Define, for any $\gamma \in (0,1)$,
\[
\phi^\epsilon(t,x) = Ce^{-\delta t}(1 + \epsilon x^2).
\]
For $C$ sufficiently large, we have that $w(0,x) < \phi^\epsilon(0,x)$. We claim that $w(t,x) < \phi^\epsilon(t,x)$ for all $t$.
Let $t_1 = \inf \lbrace t : w(t,x) < \phi^\epsilon(t,x) \text{ for all } x \rbrace$.
If $t_1 = \infty$, then we are done. Suppose now that $t_1 < \infty$.
By continuity, $t_1 > 0$. Since $w$ is bounded uniformly in time but $\phi^\epsilon(t,x) \rightarrow \infty$ as $|x| \rightarrow \infty$, there must
in fact exist a crossing point $x_1$ such that $w(t_1,x_1) = \phi^\epsilon(t_1,x_1)$ and $w(t,x) \leq \phi^\epsilon(t,x)$ for all $0 \leq t \leq t_1$.
Defining $W(t,x) = \phi^\epsilon(t,x)-w(t,x)$, we see that
\[
\begin{split}
&W_t + W-W_{xx} = (1-\delta-2\epsilon) C e^{-\delta t} + \epsilon (1-\delta) C e^{-\delta t} x^2 + \dot{y}(t) \bar{s}_0''(x+y(t)) , \\
&W_t(t_1,x_1) \leq 0, \ \ \ W_{xx}(t_1,x_1) \geq 0, \ \ \ W(t_1, x_1) = 0, \ \text{ and } \ \ W(t_1, x) \geq 0 .
\end{split}
\]
If we evaluate this at the crossing point $(t_1,x_1)$, then
\[
\begin{split}\label{eq:supersol}
0 &\geq W_t(t_1,x_1) \geq W_t(t_1,x_1) + 0 - W_{xx}(t_1,x_1) \\
&\geq (1-\delta-2\epsilon) C e^{-\delta t_1} + \epsilon (1-\delta) C e^{-\delta t_1} x_1^2
- \eta w(t_1, -y(t_1)) \| \bar{s}_0 \|_{W^{2,\infty}} \\
&\geq (1-\delta-2\epsilon) C e^{-\delta t_1} - \eta C e^{-\delta t_1}(1+\epsilon |y(t_1)|^2) \bar{M}_2 \\
& \geq C e^{-\delta t_1} \left( 1-\delta - 2\epsilon-\eta(1+\epsilon \eta^2 M_1^2 t_1^2) \bar{M}_2 \right).
\end{split}
\]
The inequality in the third line used that $w(t_1, -y(t_1)) \leq \phi^\epsilon(t_1, -y(t_1))$. We then observe that, in order to avoid
a contradiction, the final quantity above cannot be strictly positive. Unlike in the previous section, we will not be able to find a contradiction for
all finite $t_1$. Instead, we will prove that $t_1$ must go to infinity as $\epsilon \rightarrow 0$.\\
\\
Define $\eta_0 = (1-\delta)/(4 \bar{M}_2)$ and assume that $\eta < \eta_0$. Note that $\bar{M}_2$ is an upper bound for the stationary state,
and does not depend on either the initial data $z_0$ or $\eta$; likewise $\eta_0$ will not depend on $z_0$ (which makes the threshold uniform
for the entire problem) nor on $\eta$ (which would make the argument below circular). For any $\delta \in (0,1)$ fixed, and all $\epsilon$
small enough that $1-\delta -2\epsilon >(1-\delta)/2$, we observe that
\[
1-\delta-2\epsilon -\eta(1+\epsilon \eta^2 M_1^2 t^2)\bar{M}_2 > 0 \ \text{ for all } \ t < t_\epsilon := \frac{\epsilon^{-1/2}}{\eta M_1} .
\]
Note that $t_\epsilon$ does implicitly depend on $z_0$, through $M_1$. Regardless, we see from \eqref{eq:supersol} and the above that
any crossing point $t_1$ cannot be smaller than $t_\epsilon$. Thus
\[
w(t,x) = \partial_x z(t,x) \leq C e^{-\delta t} (1+\epsilon x^2) \ \text{ for all } \ 0 \leq t \leq t_\epsilon .
\]
Since $\epsilon>0$ can be arbitrarily close to zero and $t_\epsilon \rightarrow \infty$, we conclude that $\partial_x z(t,x) \leq C e^{-\delta t}$
for all $t$. This proves the claim.\\
\\
The argument above applies equally to $-w(t,x)$, which means $\| z \|_{W^{1, \infty}_x}(t) \leq C e^{-\delta t}$. From this we
refine \eqref{eq:max_dist} into
\begin{equation}\label{eq:max_dist_better}
|\dot{y}(t)| \leq C \eta e^{-\delta t} \ \text{ and } \ |y(t) + \bar{x}| \leq \frac{C \eta}{\delta} e^{-\delta t},
\end{equation}
for some $\bar{x} := -\int_0^\infty \dot{y}(\tau) \diff \! \tau$. For the actual inequality \eqref{eq:stable_stat}, we note that
\[
\begin{split}
\left| s_x (t,x) - \bar{s}_0' (x-\bar{x}) \right| &\leq \left| \partial_x z (t,x) \right| + \left| \bar{s}_0'(x+y(t))-\bar{s}_0'(x-\bar{x}) \right| \\
&\leq C e^{-\delta t} + \bar{M}_2 |y(t) + \bar{x}| \leq C \left( 1 + \frac{\bar{M}_2\eta}{\delta} \right) e^{-\delta t}
=: C_1 e^{-\delta t}
\end{split}
\]
\end{proof}

If we abuse notation and write $w(t,x) = \partial_x^m z(t,x)$, then we have
\[
\partial_t w(t,x) + w(t,x) - \partial_{xx}^2 w(t,x) = -\dot{y}(t) \bar{s}_0^{(m+1)}(x+y(t)).
\]
Armed now with the exponential decay of $|\dot{y}(t)|$, we can implement the same barrier argument used in the previous section
(indeed, the same reasoning that gave Corollary \ref{cor:higher_decay}) to conclude the following.

\begin{coro}\label{cor:stable_stat}
With the same assumptions as in Proposition \ref{prop:stable_stat}, the solution $(s,x_c)$ converges exponentially to a re-centered stationary
state in the strong topology. That is,
for any $\delta \in (0,1)$, there exists $C_2$, $\eta_0$, and $\bar{x}$ such that for all $\eta < \eta_0$ we have
\[
|x_c(t) - \bar{x}| \leq C_2 e^{-\gamma t} , \ \text{ and } \ \
\max_{0 \leq m \leq k} \left( \| s(t,x) - \bar{s}_0(x-\bar{x}) \|_{W^{m,\infty}_x}(t) \right) \leq C_2 e^{-\gamma t} .
\]
Here $\eta_0$ depends only on $g$ and $\delta$, but $C_2$ depends on $g$, $\delta$, and $s_0$.
\end{coro}

\section*{Acknowledgments}
The first author was partially funded by NSF grants DMS-2108209 and DMS-2408163.

\bibliographystyle{plain}
\bibliography{MeunierTarfulea}
%\nocite{*}

\end{document}